\documentclass[12pt]{article}
\usepackage{ayv_paperstyle}
    \makeatletter
    \let\@fnsymbol\@alph
    \makeatother

\newcommand{\ptd}{\mathcal{P}^1(\mathbb{T}^d)}
\newcommand{\tdxptd}{\mathbb{T}^d\times\mathcal{P}^1(\mathbb{T}^d)}
\newcommand{\ac}[2]{\mathcal{C}_{{#1},{#2}}}
\newcommand{\ptrd}{\mathcal{P}^1(\mathbb{T}^d\times \mathbb{R}^d)}
\newcommand{\inttd}{\int_{\mathbb{T}^d}}
\newcommand{\inttdtd}{\int_{\mathbb{T}^d\times \mathbb{T}^d}}
\newcommand{\intrd}{\int_{\mathbb{R}^d}}
\newcommand{\inttrd}{\int_{\mathbb{T}^d\times \mathbb{R}^d}}
\newcommand{\pc}[2]{\mathcal{P}^1(\mathcal{C}_{{#1},{#2}})}

\title{Viability theorem for deterministic mean field type control systems}
\author{Yurii Averboukh\thanks{Krasovskii Institute of Mathematics and Mechanics,\ \ \texttt{e-mail: ayv@imm.uran.ru, averboukh@gmail.com}}{ }\thanks{
		Ural Federal University}}

\date{}
\begin{document}
\maketitle

\begin{abstract}
	A mean field type control system is 	a dynamical system in the Wasserstein space describing an evolution of a large population of agents with mean-field interaction under a control of a unique decision maker. 	
	We develop the viability theorem for the mean field type control system. To this end we introduce a set of tangent elements to the given set of probabilities. Each tangent element is a distribution on the  tangent bundle of the phase space. The viability theorem for mean field type control systems is formulated in the classical way: the given set of probabilities on phase space is viable if and only if the set of tangent distributions intersects with the set of distributions feasible by virtue of dynamics. 
	
	
\end{abstract}

\section{Introduction}\label{sec:intro}

The theory of mean field type control system is concerned with a control problem for a large population of agents with mean-field interaction governed by a  unique decision maker. This topic is closely related with the theory of mean field games  proposed by Lasry and Lions
in~\cite{Lions01},~\cite{Lions02} and simultaneously by Huang, Caines and Malham\'{e}~\cite{Huang5}.  The mean field game theory studies the Nash equilibrium for the large population of independent agents. The similarities and differences between mean field games and mean field type control problems are discussed~in~\cite{Bensoussan_Frehse_Yam_book},~\cite{Carmona_Delarue_Lachapelle_MFG_vs_MF_control}.

The study of mean field type control systems started with  paper~\cite{ahmed_ding_controlld}. Now the mean field type control systems are examined with the help of the classical methods of the optimal control theory. The existence theorem for optimal controls is proved in~\cite{Bahlali_Mezerdi_Mezerdi_existence}.  An analog of Pontryagin maximum principle is obtained in~\cite{Andersson_Djehiche_2011},~\cite{Bensoussan_Frehse_Yam_book},~\cite{Buckdahn_Boualem_Li_PMP_SDE},~\cite{Carmona_Delarue_PMP} (see, also, \cite{Pogodaev} where the case of system with no interaction between agents is studied). Papers ~\cite{Bayraktar_Cosso_Pham_randomized},~\cite{Bensoussan_Frehse_Yam_book},~\cite{Lauriere_Pironneau_DPP_MF_control},~\cite{Pham_Wei_2015_DPP_2016} are concerned with the dynamical programming for mean field type control systems. It is well known that the dynamic programming principle leads to Bellman equation. For the mean field type control problems the Bellman equation is a partial differential equation on the space of probabilities~\cite{Bensoussan_Frehse_Yam_book},~\cite{Bensoussan_Frehse_Yam_equation},~\cite{Carmona_Delarue_master_brief}. Results of~\cite{Pham_Wei_2015_Bellman} states that the value function of the optimal control problem for mean field type control system is a viscosity solution of the Bellman equation.  The link between the minimum time function and the viscosity solutions of the corresponding Bellman equation for the  special case when the dynamics of each agent is deterministic and depends only on her state is derived in ~\cite{Marigonda_et_al_2015}. 

The viability theory provides a different tool to study optimal control problems (see~\cite{Aubin_new},~\cite{Subb_book} and references therein). In particular, for systems governed by ordinary differential equations the epigraph and hypograph of the value function are viable under certain differential inclusions~\cite{Subb_book}. Now the viability theory is developed for the wide range of dynamical systems (see~\cite{Aubin},~\cite{Aubin_new},~\cite{Aubin_Cellina} and reference therein). The key result of the viability theory is the reformulation of the viability property in the terms of tangent vectors. In particular, this theorem implies the description of the value function of optimal control problem via directional derivatives, whereas the viscosity solutions are formulated using sub- and superdifferentials. We refer to~\cite{Subb_book} for the equivalence between these two approaches for systems governed by ordinary differential equations. 

Actually,  the viability theorem for the dynamical systems in the Wasserstein space was first proved in~\cite{Viability_wasserstein_probabilistic}. The system examined in that paper arises in the optimal control problem with the  probabilistic knowledge of initial
condition. It is described by the linear Liouville equation. The viability theorem proved in~\cite{Viability_wasserstein_probabilistic} relies on embedding of the probabilities into the space of random variables and it is formulated via normal cones.

In the paper we prove the viability theorem for the deterministic mean field type control system of the general form for the case when the phase space of each agent is the torus. To this end we introduce a set of tangent elements to the given set of probabilities.  Each tangent element is a distribution on the  tangent bundle of the phase space. The viability theorem for mean field type control systems is formulated in the classical way: the given set of probabilities on phase space is viable if and only if the set of tangent distributions intersects with the set of distributions feasible by virtue of the dynamics. 

Our concept of tangency is close to the notion of geometric tangent space to the Wasserstein space introduced in \cite{Gigli}. It  was studied in \cite{Gigli}, \cite{Lott}. In those papers the relation between the geometric tangent space and the `space
of gradients' (see \cite[Definition 8.4.1]{Ambrosio} for details) is derived.

Notice that for the Banach case the notions of set of tangent vectors (tangent cone) and   subdifferential to a real-valued functions are closely related~\cite{Mordukhovich}. The subdifferential to a real-valued function defined on the Wasserstein space is introduced in~\cite[\S 10.3]{Ambrosio}. The link between this subdifferential and the set of tangent distributions introduced in the paper is the subject of the future research.

The paper is organized as follows. In  Section~\ref{sec:preliminaries} we introduce the general notations. The examined class of the dynamical systems is presented in Section~\ref{sec:MFDI}. The viability theorem is formulated in Section~\ref{sec:statement}. The auxiliary lemmas are introduced in Section~\ref{sec:lemmas}. Sufficiency and necessity parts of the viability theorem are proved in Sections~\ref{sect:sufficiency} and~\ref{sect:necissity} respectively.

\section{Preliminaries}\label{sec:preliminaries}

Given  a metric space $(X,\rho_X),$ a set $K\subset X$, $x_*\in X$, and $a\geq 0$ denote by $B_a(x_*)$ the ball of radius $a$ centered in $x_*$. If $X$ is a normed space and $x_*$ is the origin, we  write simply $B_a$ instead of $B_a(0)$. Further, denote
$$\mathrm{dist}(x_*,K)\triangleq \inf\{\rho_X(x_*,x):x\in K\}. $$

If $(X,\rho_X)$ is a separable metric space, then denote by $\mathcal{P}^1(X)$ the set of probabilities $m$ on $X$ such that, for some (and, consequently, for all) $x_*\in X$, 
$$\int_X \rho_X(x,x_*)m(dx)<\infty. $$ If $m_1,m_2\in\mathcal{P}^1(X)$, then define $1$-Wasserstein metric by the rule:
\begin{equation}\label{notion:wasser}
\begin{split}
W_1(m_1,m_2)&=\inf\left\{\int_{X\times X}\rho_X(x_1,x_2)\pi(d(x_1,x_2)):\pi\in\Pi(m_1,m_2)\right\}\\&=
\sup\left\{\int_X\phi(x)m_1(dx)-\int_X\phi(x)m_2(dx):\phi\in\mathrm{Lip}_1(X)\right\}.
\end{split}
\end{equation} 
Here $\Pi(m_1,m_2)$ is the set of plans between $m_1$ and $m_2$, i.e. \begin{equation*}\begin{split}\Pi(m_1,m_2)\triangleq \{\pi&\in\mathcal{P}^1(X\times X): \pi(A\times X)=m_1(A),\\ &\pi( X\times A)=m_2(A)\mbox{ for any mesurable } A\subset X\},\end{split} \end{equation*}
$\mathrm{Lip}_\varkappa(X)$ denotes the set of $\varkappa$-Lipschitz continuous functions on $X$.

If $\pi\in \mathcal{P}^1(X\times Y)$, where $(Y,\rho_Y)$ is a separable metric space, then denote by $\pi(\cdot|x)$ a conditional probability on $Y$ given $x$ that is a weakly measurable mapping $x\mapsto\pi(\cdot|x)\in \mathcal{P}^1(Y)$ obtained by disintegration of $\pi$ along its marginal on $X$.

If $(\Omega_1,\mathcal{F}_1)$, $(\Omega_2,\mathcal{F}_2)$ are measurable spaces, $m$ is a probability on $(\Omega_1,\mathcal{F}_1)$, $h:\Omega_1\rightarrow\Omega_2$ is measurable, then denote by $h_\#m$ a probability on $(\Omega_2,\mathcal{F}_2)$ given by the rule: for any $A\in\mathcal{F}_2$,
$$(h_\#m)(A)\triangleq m(h^{-1}(A)). $$

For simplicity we assume that the phase space is the $d$-dimensional torus $\mathbb{T}^d=\mathbb{R}^d/\mathbb{Z}^d$. Recall that the tangent space to $\td$ is $\rd$. 

Let $\mathcal{C}_{s,r}$ denote $C([s,r];\mathbb{T}^d)$.
Note that 
\begin{equation}\label{estima:wasser}
W_1(e_t{}_\#\chi_1,e_t{}_\#\chi_2)\leq W_1(\chi_1,\chi_2).
\end{equation}

If $G:[s,r]\rightrightarrows \rd$, then denote by $\int_s^r G(t)dt$ the Aumann integral of $G$ i.e.
$\int_s^r G(t)dt$ is the set of all integrals  $\int_s^rg(t)dt$ of  integrable functions $g:[s,r]\rightarrow\mathbb{R}^d $ such that $g(t)\in G(t)$.

\section{Mean field differential inclusions}\label{sec:MFDI}
This paper in concerned with the mean field type control problem for deterministic case. This is a dynamical system on a space of probabilities, where the state of the system is given by the probability $m(t)$ obeying the following equation: for all $\phi\in C(\td)$,
$$\frac{d}{dt}\inttd \phi(x)m(t,dx)=\langle f(x,m(t),u(t,x)),\nabla \phi(x)\rangle m(t,dx). $$ Here $u(t,x)$ is a control policy. 

This equation can be rewritten in the operator form
\begin{equation}\label{sys:control}
\frac{d}{dt}m(t)=\langle f(\cdot,m(t),u(t,\cdot)),\nabla\rangle m(t),
\end{equation}

Control system~(\ref{sys:control}) describes the evolution of a large population of agents when the dynamics of each agent is given by 
\begin{equation}\label{sys:control_x}
\frac{d}{dt}x(t)=f(x(t),m(t),u(t)).
\end{equation}

There are two ways of the relaxation of the control problem.  The first approach relies on measure-valued control. For mean field control systems, it was developed in several papers. Within the framework of this approach the existence result of the optimal control problem is obtained~\cite{Bahlali_Mezerdi_Mezerdi_existence}. Additionally, this approach permits the study of the limit of many particle systems~\cite{Lacker_limit}. 
We will use the second approach. It is more convenient in the viewpoint of the viability theory. The main idea of the second approach is to replace the original control system with the corresponding differential inclusion. Applying this  method to the mean field type control system, we formally replace  system~(\ref{sys:control}) with the mean field type differential inclusion (MFDI)
\begin{equation}\label{sys:mftdi}
\frac{d}{dt}m(t)\in\langle F(\cdot,m(t)),\nabla\rangle m(t).
\end{equation} Here $F(x,m)\triangleq \mathrm{co}\{f(x,m,u):u\in U\}$, symbol $\cdot$ stands for the state variable.

\begin{definition}\label{def:solution_MFDI}
	We say that the function $[0,T]\ni t\mapsto m(t)\in\mathcal{P}^1(\mathbb{T}^d)$ is a solution to~(\ref{sys:mftdi}),  if there exists a probability $\chi\in\mathcal{P}^1(\mathcal{C}_{0,T})$ such that
	\begin{enumerate}
		\item $m(t)=e_t{}_\#\chi$;
		\item any $x(\cdot)\in\mathrm{supp}(\chi)$ is absolutely continuous and,  for a.e. $t\in [0,T]$, 
		\begin{equation}\label{diff_incl:F}
		\dot{x}\in F(x(t),m(t)). 
		\end{equation}
	\end{enumerate}
\end{definition}
\begin{remark}
	The introduced definition of the solutions to the mean field type differential inclusion corresponds to the control problem for a large population of agents. It includes the solutions   defined by selectors of right-hand side of~(\ref{sys:mftdi}). This means that if the flow of probabilities  	$[0,T]\ni t\mapsto m(t) \mathcal{P}(\td)$ is such that there exists a function $w:[0,T]\times\td\rightarrow \rd$ satisfying the following properties
	\begin{itemize}
		\item  $w(t,x)\in F(x,m(t))$,
		\item  $\forall \phi\in C^1([0,T]\times\td)$  $$\int_0^T\inttd\left[\frac{\partial \phi(t,x)}{\partial t}+\langle w(t,x),\nabla\phi(t,x)\rangle\right]m(t,dx)dt=0, $$
	\end{itemize} then
	by~\cite[Theorem 8.2.1]{Ambrosio} $m(\cdot)$ solves~(\ref{sys:mftdi}) in the sense of  Definition~\ref{def:solution_MFDI} under weak assumptions on $f$ and $U$. 
\end{remark}
\begin{remark}
	There is a natural link between the solution of MFDI~(\ref{sys:mftdi}) and the relaxed controls of~(\ref{sys:control}).  
	Recall that a relaxed controls for a system described by a ordinary differential equation is a probability  $\alpha$ on $[0,T]\times U$ with the marginal on $[0,T]$ equal to Lebesgue measure. Denote by $\mathcal{U}$ the set of relaxed controls.  Given flow of probabilities $m(\cdot)$, initial state $y\in\td$ and relaxed control $\alpha\in \mathcal{U}$ denote by $x[\cdot,m(\cdot),y,\alpha]$ the solution of the equation
	\begin{equation}\label{sys:control_x_alpha}
	x(t)=y+\int_{[0,T]\times U}f(x(\tau),m(\tau),u)\mathbf{1}_{[0,t]}(\tau)\alpha(d(\tau,u)). 
	\end{equation} The function $x[\cdot,m(\cdot),y,\alpha]$ is a motion of the system ~(\ref{sys:control_x}) generated by the relaxed control $\alpha$. Further, let $\varsigma$ be a probability on $\rd\times \mathcal{U}$. We say that $[0,T]\ni t\mapsto m(t)\in\ptd$ is a flow of probabilities generated by $\varsigma$ if the marginal distribution of $\varsigma$ on $\rd$ is equal to $m(0)$ and, for any $t\in [0,T]$,
	\begin{equation}\label{eq:m_remark}
	m(t)=x[t,m(\cdot),\cdot,\cdot]_\#\varsigma. 
	\end{equation} If the existence and uniqueness theorem for~(\ref{sys:control_x_alpha}) holds true, then the solutions to~(\ref{sys:control}) determined by~(\ref{eq:m_remark})  is equivalent to the deterministic variant of the definition of  solutions to the controlled McKean-Vlasov equation proposed in~\cite{Lacker_limit}.
	
	Using~\cite[Theorem VI.3.1]{Warga}, one can prove under the conditions imposed below that $m(\cdot)$ is a flow of probabilities generated by a certain distribution of relaxed controls $\varsigma$, if and only if $m(\cdot)$ is a solution to MFDI~(\ref{sys:mftdi}).
\end{remark}

We put the following conditions:
\begin{enumerate}
	\item $F(x,m)=\mathrm{co}\{f(x,m,u):u\in U\}$, where $f$ is a continuous function defined on $\tdxptd\times U$ with values in $\mathbb{R}^d$;
	\item $U$ is compact;
	\item there exists a constant $L$ such that, for all $x_1,x_2\in\td$, $m_1,m_2\in\ptd$, $u\in U$,
	$$\|f(x_1,m_1,u)-f(x_2,m_2,u)\|\leq L(\|x_1-x_2\|+W_1(m_1,m_2)). $$
\end{enumerate}
Note that since $\td$, $\ptd$ are compact and the function $f$ is continuous, one can find a constant $R$ such that, for any  $v\in F(x,m)$, $x\in\td$, $m\in\ptd$,
\begin{equation}\label{estima:F_bound}
\|v\|\leq R.
\end{equation} 
Further, for any $v,v'\in\rd$, $x,x'\in\td$, $m,m'\in \ptd$,
\begin{equation}\label{estima:F_Lipshitz}
\begin{split}
|\mathrm{dist}(v,F(x,m))&-\mathrm{dist}(v',F(x',m'))|\\&\leq \|v-v'\|+L(\|x-x'\|+W_1(m,m')).
\end{split}
\end{equation} Additionally, if $s,r\in [0,T]$, $s<r$, $v,v'\in \rd$, $x(\cdot),x'(\cdot):[s,r]\rightarrow\td$, $m(\cdot),m'(\cdot):[s,r]\rightarrow\ptd$ are integrable, then
\begin{equation}\label{estima:dist_F_int_Lipshitz}\begin{split}\Bigl|\mathrm{dist}\Bigl(v,&\int_s^rF(x(t),m(t)dt\Bigr)- \mathrm{dist}\Bigl(v',\int_s^rF(x'(t),m'(t)dt\Bigr)\Bigr|\\ &\leq \|v-v'\|+L\int_s^r(\|x(t)-x'(t)\|+W_1(m(t),m'(t)))dt.\end{split} \end{equation}

Under the imposed conditions, one can prove that, for any $m_0\in\ptd$, and any $T>0$, there exists at least one flow of probabilities  $m(\cdot)$ solving MFDI~(\ref{sys:mftdi}) on $[0,T]$ such that $m(0)=m_0$.

\section{Statement of the Viability theorem}\label{sec:statement}

\begin{definition}\label{def:viable}
	We say that $K\subset \ptd$ is viable under MFDI~(\ref{sys:mftdi})  if, for any $m_0\in K$, there exist $T>0$ and a solution to MFDI~(\ref{sys:mftdi}) on $[0,T]$ $m(\cdot)$ such that $m(0)=m_0$, and $m(t)\in K$ for all $t\in[0,T]$.
\end{definition}
To characterize the viable sets we introduce the notion of tangent probability to a set (see Definition~\ref{def:tangent} below).

To this end denote by $\mathcal{L}(m)$ the set of probabilities $\beta$ on $\td\times \rd$ such that its marginal distribution on $\td$ is equal to $m$ and $$\inttrd\|v\|\beta(d(x,v))<\infty.$$ 
\begin{proposition} The set $\mathcal{L}(m)$ is  closed in $\mathcal{P}^1(\td\times\rd)$.
\end{proposition}
\begin{proof}
	Let $\beta_n\in\mathcal{L}(m)$, and $W_1(\beta_n,\beta)\rightarrow 0$ as $n\rightarrow \infty$. We have that, for any $\phi\in C(\td)$,
	\begin{equation}\label{equal:phi_beta_n}
	\int_{\td\times\rd}\phi(x)\beta_n(d(x,v))=\int_{\td}\phi(x)m(dx).
	\end{equation} Since, $\{\beta_n\}$ narrowly converges to $\beta$, passing to the limit in (\ref{equal:phi_beta_n}), we get that
	$$\beta\in \mathcal{L}(m). $$
\end{proof}

\begin{remark}
	One can introduce  a special  metric on $\mathcal{L}(m)$ in the following way. Let $\beta_1,\beta_2\in\mathcal{L}(m)$, denote by $\Gamma(\beta_1,\beta_2)$ the set of probabilities $\gamma$ on $\td\times \rd\times\rd$ such that, for any measurable $A\subset\td$, $C_1,C_2\subset\rd$, the following equalities hold true:
	$$\gamma(A\times C_1\times\rd)=\beta_1(A\times C_1),\ \     
	\gamma(A\times\rd\times C_2)=\beta_2(A\times C_2). $$
	Define $\mathcal{W}(\beta_1,\beta_2)$ by the rule
	\begin{equation}\label{intro:metric_beta}
	\mathcal{W}(\beta_1,\beta_2)\triangleq \inf\left\{\int_{\td\times \rd\times\rd}\|v_1-v_2\|\gamma(d(x,v_1,v_2)):\gamma\in \Gamma(\beta_1,\beta_2)\right\}.
	\end{equation}
	
	Note that the metric $\mathcal{W}$ is analogous to one introduced in \cite[Definition 5.1]{Gigli}.
	
	Using the  methods of \cite[Proposition 7.1.5]{Ambrosio} and \cite[Proposition 5.2]{Gigli}, one can prove that
	\begin{enumerate}
		\item 	$\mathcal{W}$ is a metric on $\mathcal{L}(m)$;
		\item $\mathcal{L}(m)$ with metric $\mathcal{W}$ is complete and separable;
		\item $$\mathcal{W}(\beta_1,\beta_2)=\int_{\td}W_1(\beta_1(\cdot|x),\beta_2(\cdot|x))dx; $$ here  $\beta_i(\cdot|x)$ denotes the disintegration of $\beta_i$ w.r.t. projection on $\td$.
	\end{enumerate}
	
	Notice that the topology on $\mathcal{L}(m)$ induced by $\mathcal{W}$ is stronger than the topology induced by $W_1$.
\end{remark}

Further, for $\tau>0$, define the operator $\Theta^\tau:\td\times\rd\rightarrow \td$ by the rule: for $(x,v)\in\td\times\rd$, 
\begin{equation}\label{def:Theta}
\Theta^\tau(x,v)\triangleq x+\tau v. 
\end{equation}
If $\beta\in\mathcal{L}(m)$, then $\Theta^\tau{}_\#\beta$ is a shift of $m$ through $\beta$.

\begin{definition}\label{def:tangent} Let $a>0$.
	We say that $\beta\in\mathcal{L}(m)$ is a tangent probability to $K$ at $m\in\ptd$ with the radius $a$, if there exist  sequences $\{\tau_n\}_{n=1}^\infty\subset (0,+\infty)$, $\{\beta_n\}_{n=1}^\infty\subset \mathcal{L}(m)$ such that $\mathrm{supp}(\beta_n)\subset \td\times B_a$ and
	$$\frac{1}{\tau_n}\mathrm{dist}(\Theta^{\tau_n}{}_\#\beta_n,K)\rightarrow 0,\ \ W_1(\beta_n,\beta)\rightarrow 0,\ \  \tau_n\rightarrow 0\mbox{ as }n\rightarrow\infty.$$ 
	
	Let us denote the set of tangent probabilities with the radius $a$ to $K$  by $\mathcal{T}^a_K(m)$.
\end{definition}
\begin{remark} For $\lambda\in\mathbb{R}$, let  the  rescaling operation $\mathrm{S}^\lambda:\td\times\rd\rightarrow\td\times\rd$ map a pair $(x,v)$ to
	$(x,\lambda v). $ Note that $\mathrm{S}^{\lambda_1}\mathrm{S}^{\lambda_2}=\mathrm{S}^{\lambda_1\lambda_2}$. 
	Define the scalar multiplication on $\mathcal{L}(m)$ by the rule: $$\lambda\cdot\beta\triangleq\mathrm{S}^\lambda{}_\#\beta.$$
	
	Under this definition the set of all tangent probabilities $\cup_{a>0}\mathcal{T}_K^a(m)$ becomes a cone.
	
	Indeed, for $\lambda> 0$, the mapping $\beta\mapsto \mathrm{S}^\lambda{}_\#\beta$ is a one-to-one transform of $\mathcal{L}(m)$. Furthermore, for any positive numbers $\tau$ and $\lambda$,
	$$\Theta^{\tau/\lambda}{}_\#(\mathrm{S}^\lambda{}_\#\beta)=\Theta^\tau{}_\#\beta. $$
	Thus, if $\beta\in\mathcal{T}_K^a(m)$, $\lambda>0$, then  $\lambda\cdot\beta=\mathrm{S}^\lambda{}_\#\beta\in \mathcal{T}_K^{\lambda a}(m)$.
\end{remark}

\begin{remark}
	Generally,  given $K\subset\ptd$, $m\in\ptd$, $\beta\in\mathcal{T}_K^a(m)$, one can not  find a function $w:\td\rightarrow\rd$ such that
	\begin{equation}\label{repres:beta}
	\beta(d(x,v))=w(x)m(dx)dv,
	\end{equation}  i.e. there is no embedding of the set $\mathcal{T}_K^a(m)$ into the set of measurable functions on $\td$ with valued on $\rd$.
	Indeed, let $d=1$,  $K=\{(\delta_{1/2-t}+\delta_{1/2+t})/2:t\in [0,\varepsilon]\}$. Here $\delta_\xi$ stands for the Dirac measure concentrated at $\xi$. In this case,
	$$\mathcal{T}_{K}(\delta_{1/2})=\{(\delta_{(1/2,-1)}/2+\delta_{(1/2,+1)})/2\}$$ and representation~(\ref{repres:beta}) does not hold true.  
\end{remark}

Denote by $\mathcal{F}(m)$ the set of probabilities $\beta\in\mathcal{L}(m)$ such that 
$$\inttrd\mathrm{dist}(v,F(x,m))\beta(d(x,v))=0. $$

\begin{theorem}[Viability theorem]\label{th:viability}
	A closed set $K\subset \ptd$ is viable under MFDI~(\ref{sys:mftdi}) if and only if, there exists a constant $a>0$ such that, for any $m\in K$, \begin{equation}\label{prop:nonemptyness}
	\mathcal{T}_K^a(m)\cap\mathcal{F}(m)\neq\varnothing. 
	\end{equation}
\end{theorem}
The  Viability theorem is proved in Sections~\ref{sect:sufficiency},~\ref{sect:necissity}. The proof relies on auxiliary constructions and lemmas introduced in the next section.

\section{Properties of tangents probabilities}\label{sec:lemmas}
Let $(X_1,\rho_1)$, $(X_2,\rho_2)$, $(X_3,\rho_3)$ be separable metric spaces. Let $\pi_{1,2}$, $\pi_{2,3}$ be probabilities on $X_1\times X_2$ and $X_2\times X_3$, respectively. Assume that $\pi_{1,2}$ and $\pi_{2,3}$ have the same  marginal distributions on $X_2$. 
Define the probability $\pi_{1,2}*\pi_{2,3}\in\mathcal{P}(X_1\times X_3)$ by the rule: for all $\phi\in C_b(X_1\times X_3)$,
\begin{equation*}\begin{split} 
\int_{X_1\times X_3}\phi(x_1,x_3)\pi_{1,2}*&\pi_{2,3}(d(x_1,x_3)) \\ 
&\triangleq \int_{X_1\times X_2}\int_{X_3}\phi(x_1,x_3)\pi_{2,3}(dx_3|x_2)\pi_{1,2}(d(x_1,x_2)). 
\end{split}\end{equation*}
The operation $(\pi_{1,2},\pi_{2,3})\mapsto \pi_{1,2}*\pi_{2,3}$ is a composition of probabilities. In~\cite{Ambrosio} it is denoted by $\pi_{2,3}\circ\pi_{1,2}$ due to the natural analogy with the composition of functions. However, we prefer the designation $\pi_{1,2}*\pi_{2,3}$ because it explicitly points out  the marginals of the compositions of probabilities.


\begin{remark} If $(X_4,\rho_4)$ is a metric space, $\pi_{3,4}$ is a probability on $X_3\times X_4$ such that marginal distributions of $\pi_{2,3}$ and $\pi_{3,4}$ on $X_3$ coincides, then
	$$(\pi_{1,2}*\pi_{2,3})*\pi_{3,4}=\pi_{1,2}*(\pi_{2,3}*\pi_{3,4}). $$
\end{remark}

Note that if $\pi_{m',m}$ is a plan between $m'$ and $m$, $\beta\in\mathcal{L}(m)$, then $\pi_{m',m}*\beta\in \mathcal{L}(m')$.

\begin{lemma}\label{lm:transfer_Theta} If $\tau>0$, $m,m'\in\ptd$, $\pi_{m',m}\in \Pi(m',m)$ is an optimal plan between $m'$ and $m$,  $\beta\in\mathcal{L}(m)$, then
	$$W_1(\Theta^\tau{}_\#\beta,\Theta^\tau{}_\#(\pi_{m',m}*\beta))\leq W_1(m',m). $$
\end{lemma}
\begin{proof}
	Let $\phi\in \mathrm{Lip}_1(\td)$. We have that
	\begin{equation*}
	\begin{split}
	\inttd&\phi(y')(\Theta^\tau{}_\#(\pi_{m',m}*\beta))(dy')- \inttd\phi(y)(\Theta^\tau{}_\#\beta)(dy)\\&=
	\inttrd \phi(x'+\tau v)(\pi_{m',m}*\beta)(d(x',v))-\inttrd \phi(x+\tau v)\beta(d(x,v))\\&=
	\inttdtd\intrd[\phi(x'+\tau v)-\phi(x+\tau v)]\beta(dv|x)\pi_{m',m}(d(x',x))\\ &\leq 
	\inttdtd\intrd\|x'-x\|\beta(dv|x)\pi_{m',m}(d(x',x))=W_1(m',m).
	\end{split}
	\end{equation*} This fact, together with the definition of 1-Wasserstein metric, imply the conclusion of the lemma. 
\end{proof}

\begin{lemma}\label{lm:distance_F}
	Let $m,m'\in\ptd$, $\pi_{m',m}\in \Pi(m',m)$ be an optimal plan between $m'$ and $m$,  $\beta\in\mathcal{L}(m)$. Then
	\begin{equation*}\begin{split}\Bigl|\inttrd\mathrm{dist}(v,&F(x,m))\beta(d(x,v))\\- \inttrd&\mathrm{dist}(v,F(x,m'))(\pi_{m',m}*\beta)(d(x',v))\Bigr| \leq 2LW_1(m',m). \end{split}\end{equation*}
\end{lemma}
\begin{proof}
	From (\ref{estima:F_Lipshitz}) we obtain
	\begin{multline*}
	\Bigl|\inttrd\mathrm{dist}(v,F(x,m))\beta(d(x,v))\\- \inttrd\mathrm{dist}(v,F(x',m'))(\pi_{m',m}*\beta)(d(x',v))\Bigr|\\\leq
	\inttdtd\intrd|\mathrm{dist}(v,F(x,m))-\mathrm{dist}(v,F(x',m'))|\beta(dv|x)\pi_{m'm}(d(x',x))\\ \leq L\inttdtd\intrd(\|x'-x\|+W_1(m,m'))\beta(dv|x)\pi_{m'm}(d(x',x))\\=  2LW_1(m',m).
	\end{multline*}
\end{proof}

The following lemma is a cornerstone of the sufficiency part of the Viability theorem. It is analogous to~\cite[Lemma 3.4.3]{Aubin}.
\begin{lemma}\label{lm:one_step} Assume that $K\subset\ptd$ is compact and~(\ref{prop:nonemptyness}) is fulfilled. Then, for each  natural $n$, one can find a number $\theta_n\in (0,1/n)$ such that, for any $m\in K$, there exist $s\in (\theta_n,1/n)$, $\beta\in\mathcal{L}(m)$ and $\nu\in K$ satisfying the following properties:
	\begin{enumerate}
		\item $W_1(\Theta^s{}_\#\beta,\nu)<s/n$;
		\item $$\inttrd\mathrm{dist}(v,F(x,m))\beta(d(x,v))<1/n; $$
		\item $\mathrm{supp}(\beta)\subset \td\times B_a$.
	\end{enumerate}
\end{lemma}
\begin{proof}

	First,  we claim that, given probability $\mu\in K$, and natural $n$, there exist a time $r_\mu\in (0,1/n)$ and a probability $\hat{\beta}_\mu\in \mathcal{L}(\mu)$ such that 
	\begin{equation}\label{ineq:dist_beta_mu}
	\mathrm{dist}(\Theta^{r_\mu}{}_\#\hat{\beta}_\mu,K)<\frac{r_\mu}{2n};
	\end{equation}
	\begin{equation}\label{ineq:int_dist}
	\int_{\td\times \rd}\mathrm{dist}(v,F(x,\mu))\hat{\beta}_\mu(d(x,v))<\frac{1}{2n}; 
	\end{equation}
	\begin{equation}\label{incl:supp_beta_mu}
	\mathrm{supp}(\hat{\beta}_\mu)\subset \td\times B_a.
	\end{equation}	
	Indeed, given $\beta\in\mathcal{T}^a_K(\mu)\cap \mathcal{F}(\mu)$, one can choose $r_\mu\in (0,1/n)$ and $\hat{\beta}_\mu$ such that (\ref{ineq:dist_beta_mu}) and (\ref{incl:supp_beta_mu}) are fulfilled and 
	$$W_1(\hat{\beta}_\mu,\beta)<\frac{1}{2n\cdot\max\{L,1\}}.$$ 
	Since the function $(x,v)\mapsto \mathrm{dist}(v;F(x,\mu))$ is Lipschitz continuous for the constant $\max\{L,1\}$ (see (\ref{estima:F_Lipshitz})), we get inequality (\ref{ineq:int_dist}).
	
	Let $\mathcal{E}_n(\mu)$ be a subset of $\ptd$  such that, for any $m\in\mathcal{E}_n(\mu)$, there exists a probability 
	$\beta\in\mathcal{L}(m) $ satisfying the following conditions:
	\begin{list}{(E\arabic{tmp})}{\usecounter{tmp}}
		\item $\mathrm{dist}(\Theta^{r_\mu}{}_\#\beta,K)<r_\mu/n;$
		\item $$\inttrd\mathrm{dist}(v,F(x,m))\beta(d(x,v))<1/n; $$
		\item $\mathrm{supp}(\beta)\subset \td\times B_a$.
	\end{list}
	
	Properties (\ref{incl:supp_beta_mu})--(\ref{ineq:int_dist}) yield that $\mu$ belongs to $\mathcal{E}_n(\mu)$. Thus, 
	\begin{equation}\label{incl:covering_E}
	K\subset \bigcup_{\mu\in K}\mathcal{E}_n(\mu).
	\end{equation}
	
	Now we show that each set $\mathcal{E}_n(\mu)$ is open. To this end we prove that, for any $m\in \mathcal{E}_n(\mu)$, one can find a positive constant $\varepsilon$ depending on $n$, $\mu$ and $m$ such that $B_\varepsilon(m)\subset\mathcal{E}_n(\mu)$. First, observe that since $m\in \mathcal{E}_n(\mu)$, there exists $\beta\in\mathcal{L}(m)$ satisfying conditions (E1)--(E3). Now let $m'\in \ptd$.
	
	Put \begin{equation}\label{intro:beta_prime}
	\beta'\triangleq \pi_{m',m}*\beta,
	\end{equation}  where $\pi_{m',m}$ is an optimal plan between $m'$ and $m$.
	We have that $\beta'\in \mathcal{L}(m')$. 
	Lemma~\ref{lm:transfer_Theta} yields that 
	\begin{equation}\label{ineq:beta_prime_transfer}
	\begin{split}
	\mathrm{dist}(\Theta^{r_\mu}{}_\#\beta',K)
	\leq W_1(\Theta^{r_\mu}{}_\#\beta',\Theta^{r_\mu}{}_\#\beta)&+\mathrm{dist}(\Theta^{r_\mu}{}_\#\beta,K)\\ \leq W_1(m',&m)+\mathrm{dist}(\Theta^{r_\mu}{}_\#\beta,K).\end{split}
	\end{equation}
	Further, from Lemma~\ref{lm:distance_F} it follows  that
	\begin{equation*}\begin{split}\inttdtd\mathrm{dist}(&v,F(x,m'))\beta'(d(x,v))\\&\leq 2LW_1(m',m)+\inttrd\mathrm{dist}(v,F(x,m))\beta(d(x,v)).\end{split}\end{equation*}

	This and~(\ref{ineq:beta_prime_transfer}) give that if \begin{equation*}\begin{split}
	W(m',m')<\varepsilon \triangleq\min\Bigl\{&\frac{1}{n}-\mathrm{dist}(\Theta^{r_\mu}{}_\#\beta,K),
	\\&\frac{1}{2Ln}-\frac{1}{2L} \inttrd\mathrm{dist}(v,F(x,m))\beta(d(x,v))\Bigr\},\end{split} \end{equation*} then conditions (E1) and (E2) are fulfilled for $\beta'$. Furthermore, condition (E3) holds true for $\beta'$ by~(\ref{intro:beta_prime}). Hence, $B_\varepsilon(m)\subset\mathcal{E}_\mu$. Therefore, the set $\mathcal{E}_n(\mu)$ is open.
	
	Since $K$ is a closed subset of the compact space $\ptd$, and $\{\mathcal{E}_n(\mu)\}_{\mu\in K}$ is an open cover of $K$, there exists a finite number of probabilities $\mu_1,\ldots,\mu_I\in K$ such that
	$$
	K\subset \bigcup_{i=1}^I\mathcal{E}_n(\mu_i). 
	$$
	Note that $r_{\mu_i}\in (0,1/n)$.
	Put $$\theta_n\triangleq\min_{i\in \overline{1,I}}r_{\mu_i}.$$
	
	Now let $m\in K$. There exists a number $i$ such that $m\in\mathcal{E}(\mu_i)$. This means that, for some $\beta\in\mathcal{L}(m)$ and $\mu=\mu_i$, conditions (E1)--(E3) hold true. To complete the proof of the lemma it suffices to put $s\triangleq r_{\mu_i}$ and to choose $\nu\in K$ to be nearest~to~$\Theta^s{}_\#\beta$. 
\end{proof}

\section{Proof of the Viability theorem. Sufficiency}\label{sect:sufficiency}
To prove the sufficiency part of the Viability theorem we introduce  the concatenation of probabilities on space of motions in the following way. First, if $x_1(\cdot)\in\mathcal{C}_{s,r}$, $x_2(\cdot)\in\mathcal{C}_{r,\theta}$ are such that $x_1(r)=x_2(r)$, then
$$(x_1(\cdot)\odot x_2(\cdot))(t)\triangleq \left\{\begin{array}{cc}
x_1(t),& t\in [s,r],\\
x_2(t),& t\in [r,\theta].
\end{array}
\right. $$ Note that $x_1(\cdot)\odot x_2(\cdot)\in \mathcal{C}_{s,\theta}$.

Now let $\chi_1\in\pc{s}{r}$, $\chi_2\in\pc{r}{\theta}$ be such that $e_r{}_\#\chi_1=e_r{}_\#\chi_2=m$. Let $\{\chi_2(\cdot|y)\}_{y\in\td}$ be a family of conditional probabilities such that, for any $\phi\in C_b(\mathcal{C}_{r,\theta})$,
$$\int_{\mathcal{C}{r,\theta}}\phi(x(\cdot))\chi_2(d(x(\cdot)))= \inttd\int_{\mathcal{C}_{r,\theta}}\phi(x(\cdot))\chi_2(d(x(\cdot))|y)m(dy). $$ Note that $\mathrm{supp}(\chi_2(\cdot|y))\subset \{x(\cdot)\in\mathcal{C}{r,\theta}:x(r)=y\}$
Finally, for $A\subset\pc{s}{\theta}$ put 
$$(\chi_1\odot\chi_2)(A)\\\triangleq \int_{\mathcal{C}_{s,r}} \chi_2(\{x_2(\cdot):(x_1(\cdot)\odot x_2(\cdot))\in A\}|x_1(r))\chi_1(d(x_1(\cdot))).$$

\begin{proof}[Proof of Theorem ~\ref{th:viability}. Sufficiency]
Given $m_0\in K$, $T>0$, and a natural number $n$, let us construct  a number $J_n$ and sequences $\{t_n^j\}_{j=0}^{J_n}\subset [0,+\infty)$, $\{\mu_n^j\}_{j=0}^{J_n}\subset\ptd$, $\{\nu_n^j\}_{j=0}^{J_n}\subset K$, $\{\beta^j_n\}_{j=1}^{J_n}\subset \ptrd$ by the following rules:
\begin{enumerate}
	\item $t_n^0\triangleq 0$, $\mu_n^0=\nu_n^0\triangleq m_0$;
	\item If $t_n^j<T$, then choose $s_n^{j+1}\in (\theta_n,1/n)$, $\beta_n^{j+1}\in\mathcal{L}(\nu_n^j)$ and $\nu_n^{j+1}\in K$ satisfying conditions of Lemma~\ref{lm:one_step} for $m=\nu_n^j$. Put $t_n^{j+1}\triangleq t_n^j+s_n^{j+1}$, $\mu_n^{j+1}\triangleq \Theta^{s_n^{j+1}}{}_\#(\pi_n^j*\beta_n^{j+1})$, where $\pi_n^j$ is an optimal plan between $\mu_n^j$ and $\nu_n^j$.
	\item If $t_n^j\geq T$, then put $J_n\triangleq j$.
\end{enumerate}
Since $t_n^{j+1}-t_n^j\geq\theta_n$, this procedure is finite. 

Now let us prove that, for $j=\overline{0,J_n}$,
\begin{equation}\label{ineq:W_1_mu_nu}
W_1(\mu_n^j,\nu_n^j)\leq t_n^j/n.
\end{equation}
For $j=0$ inequality~(\ref{ineq:W_1_mu_nu}) is fulfilled by the construction. Assume that~(\ref{ineq:W_1_mu_nu}) holds true for some~$j\in \overline{0,J_n-1}$. We have that
\begin{equation}\label{ineq:triangle_mu_nu_n_j}
\begin{split}
W_1(\mu_n^{j+1},\nu_n^{j+1})=W_1(\Theta^{s_n^{j+1}}{}_\#(\pi_n^j&*\beta_n^{j+1}),\nu_n^{j+1})\\ \leq W_1(\Theta^{s_n^{j+1}}{}_\#(&\pi_n^j*\beta_n^{j+1}),\Theta^{s_n^{j+1}}{}_\#\beta_n^{j+1})\\
&+ W_1(\Theta^{s_n^{j+1}}{}_\#\beta_n^{j+1},\nu_n^{j+1})).
\end{split}
\end{equation}
Recall that $\pi_n^j$ denotes the optimal plan between $\mu_n^j$ and $\nu_n^j$.
This, inequality~(\ref{ineq:triangle_mu_nu_n_j}), the choice of $s_n^{j+1}$, $\beta_n^{j+1}$, $\nu_n^{j+1}$ and Lemmas~\ref{lm:transfer_Theta},~\ref{lm:one_step} imply that
$$W_1(\mu_n^{j+1},\nu_n^{j+1})\leq W_1(\mu_n^{j},\nu_n^{j})+s_n^{j+1}/n. $$ Hence, using assumption, we get $$W_1(\mu_n^{j+1},\nu_n^{j+1})\leq t_n^{j+1}/n.$$ This proves~(\ref{ineq:W_1_mu_nu})

Put $$\tau_n^j\triangleq\left\{\begin{array}{ll}
t_j^n, & j=0,\ldots, J_n-1,\\
T, & j=J_n.
\end{array}\right.$$ 

For $j=\overline{1,J_n}$ define the map $\Lambda^j_n:\tdxptd\rightarrow\mathcal{C}_{t_n^{j-1},t_n^{j}}$ by the rule: 
$$(\Lambda^j_n(x,v))(t)\triangleq x+(t-\tau_n^{j-1})v,\ \ t\in[\tau_n^{j-1},\tau_n^{j}]. $$

Put $\chi_n^j\triangleq \Lambda_n^j{}_\#(\pi_n^{j-1}*\beta_n^{j})$. Note that $e_{0}{}_\#\chi_n^1=m_0$, $e_{\tau_n^{j}}{}_\#\chi_n^j=e_{\tau_n^{j}}{}_\#\chi_n^{j+1}$. Thus,
the probability $$\chi_n\triangleq \chi_n^1\odot\ldots\odot\chi_n^{J_n}$$ is well-defined. Note  that $\chi_n\in\pc{0}{T}$.

Recall that $\operatorname{supp}(\beta_n^j)\subset \td\times B_a$. Hence, if $x(\cdot)\in \mathrm{supp}(\chi_n)$, then, for all $t',t''\in [0,T]$, \begin{equation}\label{estima:dot_x_chi_n}
\|x(t')-x(t'')\|\leq a|t'-t''|. 
\end{equation}

Denote $m_n(t)\triangleq e_t{}_\#\chi_n$. Inequality~(\ref{estima:dot_x_chi_n}) yields that 
\begin{equation}\label{ineq:lip_m_n}
W_1(m_n(t'),m_n(t''))\leq a|t'-t''|.
\end{equation}
We have that
$\ m_n(t_n^j)=\mu_n^j. $ Therefore, using~(\ref{ineq:W_1_mu_nu}),~(\ref{ineq:lip_m_n}) and inclusion $\nu_n^j\in K$, we obtain that
\begin{equation}\label{ineq:dist_m_n_K}
\mathrm{dist}(m_n(t),K)\leq (T+a)/n.
\end{equation}

Given $s,r\in [0,T]$, $s<r$ let $I^0_n,I^1_n$ be such that $s\in [\tau_n^{I^0_n-1},\tau_n^{I^0_n}]$, $r\in [\tau_n^{I^1_n-1},\tau_n^{I^1_n}]$. For sufficiently large $n$, $I^0_n<I_n^1$. Put $\zeta_n^{I_0-1}\triangleq s$, $\zeta_n^{i}\triangleq \zeta_n^{i}$, $i=I_n^0,\ldots, I_n^1-1$, $\zeta_n^{I_1}\triangleq r$. For $i=I_n^0,\ldots, I_n^1$, denote $\delta^i_n\triangleq \zeta_n^i-\zeta_n^{i-1}$. 

Now assume that $x(\cdot)\in \mathrm{supp}(\chi_n)$. Using inequalities~(\ref{estima:dot_x_chi_n}),~(\ref{ineq:lip_m_n}) together with the fact that, for $t\in [\zeta_n^{i-1},\zeta_n^i]$, $|t-\tau_n^{i-1}|\leq 1/n$, we get
\begin{equation*}\label{ineq:dist_x_m_n}
\begin{split}\mathrm{dist}&\left(x(r)-x(s),\int_s^r F(x(t),m_n(t))d t\right)\\&\leq
\sum_{i=I_n^0}^{I^1_n}\mathrm{dist}\left(x(\zeta^i_n)-x(\zeta^{i-1}_n),\int_{\zeta^{i-1}_n}^{\zeta^i_n} F(x(t),m_n(t)dt\right)\\
&\leq 
\sum_{i=I_n^0}^{I^1_n}\mathrm{dist}\left(x(\zeta^i_n)-x(\zeta^{i-1}_n), \delta^i_nF(x(\tau^{i-1}_n),m_n(\tau^{i-1}_n))\right)+2(r-s)La/n.
\end{split}
\end{equation*}

Thus,
\begin{equation}\label{ineq:intC_sup_dist}
\begin{split}
\int_{\mathcal{C}_{0,T}}&\mathrm{dist}\Bigl( x(r)-x(s),\int_r^s F(x(t),m_n(t))dt\Bigr)\chi_n(dx(\cdot))\\ &\leq
\sum_{i=I_n^0}^{I^1_n}\int_{\mathcal{C}_{\zeta_n^{i-1},\zeta_n^{i}}}\mathrm{dist} \Bigl(x(\zeta^i_n)-x(\zeta^{i-1}_n), \delta^i_nF(x(\tau^{i-1}_n),m_n(\tau^{i-1}_n))\Bigr)\chi_n^i(dx(\cdot))\\&{} \hspace{274pt}+2(r-s)La/n.
\end{split}
\end{equation}
By the construction of $\chi_n^i$ we have that
\begin{equation*}
\begin{split}
\int_{\mathcal{C}_{\zeta_n^{i-1},\zeta_n^{i}}}&\mathrm{dist} \Bigl(x(\zeta^i_n)-x(\zeta^{i-1}_n), \delta^i_nF(x(\tau^{i-1}_n),m_n(\tau^{i-1}_n))\Bigr)\chi_n^i(dx(\cdot))\\&=
\inttrd \mathrm{dist}\left(\delta^i_nv, \delta^i_nF(x,\mu_n^{i-1}))\right)(\pi_n^{i-1}*\beta_n^{i})(d(x,v)) \\&=\delta^i_n\inttrd \mathrm{dist}\left(v, F(x,\mu_n^{i-1})\right)(\pi_n^{i-1}*\beta_n^{i})(d(x,v)).
\end{split}
\end{equation*}
This, Lemma \ref{lm:distance_F}, inequality (\ref{ineq:W_1_mu_nu}) and the choice of $\pi_n^{i-1}$ yield the estimate
\begin{equation*}
\begin{split}
\int_{\ac{\tau_n^{i-1}}{\tau_n^{i}}}\mathrm{dist}&\left(x(\zeta^i_n)-x(\zeta^{i-1}_n), \delta^i_nF(x(\tau^{i-1}_n),m_n(\tau^{i-1}_n))\right)\chi_n^i(dx(\cdot))\\
&\leq\delta^i_n\inttrd \mathrm{dist}\left(v, F(x,\nu^{i-1}_n)\right)\beta_n^{i}(d(x,v))+\delta^{i}_n2LT/n.
\end{split}\end{equation*}
Therefore, taking into account equality $\sum \delta^i_n=(r-s)$, inequality (\ref{ineq:intC_sup_dist}), the choice of $\beta_n^j$ and Lemma \ref{lm:one_step} we conclude that
\begin{equation}\label{estima:chi_n_distance_F}
\begin{split}
\int_{\ac{0}{T}}\mathrm{dist}\Bigl( x(r)-x(s),\int_s^r &F(x(t),m_n(t))dt\Bigr)\chi_n(dx(\cdot)) \\ &\leq 
(r-s)(1+2LT+2La)/n.
\end{split}
\end{equation}

Furthermore, we have that, for each natural $n$, $\mathrm{supp}(\chi_n)$ lie in the compact set of $a$-Lipschitz continuous function from $[0,T]$ to $\td$. By~\cite[Proposition 7.1.5]{Ambrosio} the sequence $\{\chi_n\}$ is relatively compact in $\pc{0}{T}$. This means that there exist a sequence $n_l$ and probability $\chi\in\pc{0}{T}$ such that 
$$W_1(\chi_{n_l},\chi)\rightarrow 0\mbox{ as }l\rightarrow \infty. $$  
Notice that $x(\cdot)\in \mathrm{supp}(\chi)$, then $x(\cdot)$ is $a$-Lipschitz continuous and, thus, absolutely continuous.

Put $m(t)\triangleq e_t{}_\#\chi$. Inequality~(\ref{estima:wasser}) implies that, for any $t\in [0,T]$,
\begin{equation}\label{estima:wasser_n_l}
W_1(m(t),m_{n_l}(t))\leq W_1(\chi,\chi_{n_l}). 
\end{equation}

Since the functions $\mathcal{C}_{0,T}\ni x(\cdot)\mapsto \mathrm{dist}\bigl(x(r)-x(s),\int_s^rF(x(t),m(t))dt\bigr)$ is Lipschitz continuous for the constant $(2+L(r-s))$, using~(\ref{estima:dist_F_int_Lipshitz})~and~(\ref{estima:wasser_n_l}),  we have that
\begin{equation*}
\begin{split}
\int_{\ac{0}{T}}\mathrm{dist}\Bigl(x(r)-x(s),\int_s^rF(x(t),m(t))&dt\Bigr)\chi(d(x(\cdot)))\\ \leq 
\int_{\ac{0}{T}}\mathrm{dist}\Bigl(x(r)-x(s),\int_s^rF(x(t),&m_{n_l}(t))dt\Bigr)\chi_{n_l}(d(x(\cdot)))\\+ (&2+2L(r-s))W_1(\chi,\chi_{n_l}).
\end{split}\end{equation*}
Thus, by (\ref{estima:chi_n_distance_F}) $$\int_{\ac{0}{T}}\mathrm{dist}\Bigl(x(r)-x(s),\int_s^rF(x(t),m(t))dt\Bigr)\chi(d(x(\cdot)))=0. $$ This means that, for any $x(\cdot)\in\mathrm{supp}(\chi)$ and any $r,s\in [0,T]$, $s<r$, $$x(r)-x(s)\in \int_s^rF(x(t),m(t))dt.$$ Hence, each  $x(\cdot)\in\mathrm{supp}(\chi)$ solves~(\ref{diff_incl:F}). Consequently, $m(\cdot)$ is a solution to MFDI~(\ref{sys:mftdi}).

Finally,
$$\mathrm{dist}(m(t),K)\leq W_1(m(t),m_{n_l}(t))+\mathrm{dist}(m_{n_l}(t),K). $$
This, ~(\ref{ineq:dist_m_n_K}) and~(\ref{estima:wasser_n_l}) yield that, for any $t\in [0,T]$,
$$m(t)\in K. $$

Since $m(\cdot)$ is a solution of MFDI~(\ref{sys:mftdi}), we conclude that $K$ is viable under MFDI~(\ref{sys:mftdi}). \end{proof}

\section{Proof of Viability theorem. Necessity}\label{sect:necissity}
Notice that, if $[0,T]\ni t\mapsto m(t)$ solves MFDI~(\ref{sys:mftdi}), then 
\begin{equation}\label{ineq:lip_m}
W_1(m(t'),m(t''))\leq R|t'-t''|.
\end{equation} Indeed, let $\chi\in\pc{0}{T}$ be such that $m(t)=e_t{}_\#\chi$ and, for any $x(\cdot)\in\mathrm{supp}(\chi)$, $\dot{x}(t)\in F(x(t),m(t))$ a.e. $t\in [0,T]$. Define the plan between $m(t')$ and $m(t'')$ by the rule: for $\phi\in C(\td\times\td)$, 
$$\inttdtd\phi(x',x'')\pi(d(x',x''))=\int_{\ac{0}{T}}\phi(x(t'),x(t''))\chi(d(x(\cdot))). $$ We have that
\begin{equation*}\begin{split}W_1(m(t'),m(t''))&\leq \inttdtd\|x'-x''\|\pi(d(x',x''))\\&=\int_{\ac{0}{T}}\|x(t')-x(t'')\|\chi(d(x(\cdot)))\leq R|t'-t''|.\end{split}\end{equation*}

Now define the operator $\Delta^\tau:\mathcal{C}_{0,T}\rightarrow\td\times\rd$ by the following rule:
\begin{equation}\label{def:Delta}
\Delta^\tau(x(\cdot))\triangleq \left(x(0),\frac{x(\tau)-x(0)}{\tau}\right) .
\end{equation} This operator will play the crucial role in the following.

\begin{proof}[Proof of Theorem  ~\ref{th:viability}. Necessity]
Let $m_0\in K$. By assumption, there exist a time $T$, a flow of probabilities on $[0,T]$ $m(\cdot)$ and a probability $\chi\in\pc{0}{T}$ be such that 
\begin{itemize}
	\item $m(t)=e_t{}_\#\chi$,
	\item $m(0)=m_0$,
	\item if $x(\cdot)\in\mathrm{supp}(\chi)$, then $x(\cdot)$ is absolutely continuous and $\dot{x}(t)\in F(x(t),m(t))$ a.e. $t\in [0,T]$,
	\item $m(t)\in K$.
\end{itemize} 

Put $$
\beta_\tau\triangleq \Delta^\tau{}_\#\chi.
$$

The definitions of the operators $\Theta^\tau$ and $\Delta^\tau$ (see~(\ref{def:Theta}) and~(\ref{def:Delta})) yield that
$$\Theta^\tau{}_\#\beta_\tau=m(\tau). $$ This means that
\begin{equation}\label{incl:beta_tau_K}
\Theta^\tau{}_\#\beta_\tau\in K.
\end{equation}

Further, the definition of $\beta_\tau$ implies that
\begin{equation}\label{incl:supp_beta_tau}
\mathrm{supp}(\beta_\tau)\subset \td\times B_R.
\end{equation}

Now let us prove that 
\begin{equation}\label{ineq:dist_F_beta_tau}
\int_{\td\times\rd}\mathrm{dist}(v,F(x,m_0))\beta_\tau(d(x,v))\leq LR \tau.
\end{equation}
Indeed, if $x(\cdot)$ belongs to $\mathrm{supp}(\chi)$ then it solves differential inclusion~(\ref{diff_incl:F}). In particular, $\|x(t)-x(0)\|\leq Rt$. Hence, for  $x(\cdot)\in\mathrm{supp}(\chi)$,
\begin{equation}\label{equal:x_tau_int_F}
\mathrm{dist}\left(x(\tau)-x(0),\int_0^\tau F(x(t),m(t))dt\right)=0.
\end{equation}
Using inequality (\ref{estima:dist_F_int_Lipshitz}) we obtain, for $x(\cdot)\in\mathrm{supp}(\chi)$,
\begin{equation*}\begin{split}\mathrm{dist}&(\Delta^\tau(x(\cdot)),F(x(0),m_0) )\\&=\frac{1}{\tau}\mathrm{dist}\left(x(\tau)-x(0),\int_0^\tau F(x(0),m(0))dt \right)\\&\leq \frac{1}{\tau}\mathrm{dist}\left(x(\tau)-x(0),\int_0^\tau F(x(t),m(t))dt \right)+LR\tau.\end{split}\end{equation*}
This  and (\ref{equal:x_tau_int_F}) proves~(\ref{ineq:dist_F_beta_tau}).

By inclusion~(\ref{incl:supp_beta_tau}) and \cite[Proposition 7.1.5]{Ambrosio} we conclude that
there exist a sequence $\{\tau_n\}_{n=1}^\infty$ and a probability $\beta\in\ptrd$ such that
$$\tau_n\rightarrow 0,\ \ W_1(\beta_{\tau_n},\beta)\rightarrow 0\mbox{ as } n\rightarrow \infty.$$

This and (\ref{incl:beta_tau_K}) imply that
\begin{equation}\label{incl:beta_tangent}
\beta\in\mathcal{T}_K^a(m_0)
\end{equation} for $a=R$.

Further,  passing to the limit in (\ref{ineq:dist_F_beta_tau}) we get the inclusion
$$\beta\in\mathcal{F}(m_0). $$

Combining this and (\ref{incl:beta_tangent}), we conclude that (\ref{prop:nonemptyness}) holds true for any $m\in\td$ with the constant $a$ that does not depend on $m$.
\end{proof}

\bibliography{th_28_wasserstein_a}

\end{document}